\documentclass[11pt]{amsart}
\usepackage[margin=1.35in]{geometry}
\usepackage{amscd,amsmath,amsxtra,amsthm,amssymb,stmaryrd,xr,mathrsfs,mathtools,enumerate,commath, comment, enumitem}
\usepackage{stmaryrd}
\usepackage{xcolor}
\usepackage{commath}
\usepackage{comment}
\usepackage{tikz-cd}
\usepackage{longtable} % for 'longtable' environment
\usepackage{pdflscape} % for 'landscape' environment
\usepackage{booktabs}
\usepackage{hyperref}
\definecolor{vegasgold}{rgb}{0.77, 0.7, 0.35}
\definecolor{darkgoldenrod}{rgb}{0.72, 0.53, 0.04}
\definecolor{gold(metallic)}{rgb}{0.83, 0.69, 0.22}
\hypersetup{
 colorlinks=true,
 linkcolor=darkgoldenrod,
 filecolor=brown,      
 urlcolor=gold(metallic),
 citecolor=darkgoldenrod,
 }

\usepackage[all,cmtip]{xy}

\DeclareFontFamily{U}{wncy}{}
\DeclareFontShape{U}{wncy}{m}{n}{<->wncyr10}{}
\DeclareSymbolFont{mcy}{U}{wncy}{m}{n}
\DeclareMathSymbol{\Sh}{\mathord}{mcy}{"58}
\usepackage[T2A,T1]{fontenc}
\usepackage[OT2,T1]{fontenc}

\newtheorem{theorem}{Theorem}[section]
\newtheorem{lemma}[theorem]{Lemma}

\newtheorem*{theorem*}{Theorem}
\newtheorem*{ass*}{Assumption}
\newtheorem{definition}[theorem]{Definition}

\newtheorem{remark}[theorem]{Remark}

\newtheorem{proposition}[theorem]{Proposition}

\newcommand{\Z}{\mathbb{Z}}
\newcommand{\p}{\mathfrak{p}}
\newcommand{\Q}{\mathbb{Q}}
\newcommand{\F}{\mathbb{F}}

\newcommand{\cC}{\mathcal{C}}

\newcommand{\op}[1]{\operatorname{#1}}

\newcommand\mtx[4] { \left( {\begin{array}{cc}
 #1 & #2 \\
 #3 & #4 \\
 \end{array} } \right)}

\numberwithin{equation}{section}

\begin{document}

\title[Surjective Galois representations for Drinfeld modules]{The $T$-adic Galois representation is surjective for a positive density of Drinfeld modules}

\author[A.~Ray]{Anwesh Ray}
\address[Ray]{Chennai Mathematical Institute, H1, SIPCOT IT Park, Kelambakkam, Siruseri, Tamil Nadu 603103, India}
\email{anwesh@cmi.ac.in}


\maketitle

\begin{abstract}
   Let $\F_q$ be the finite field with $q\geq 5$ elements, $A:=\F_q[T]$ and $F:=\F_q(T)$. Assume that $q$ is odd and take $|\cdot|$ to be the absolute value at $\infty$ that is normalized by $|T|=q$. Given a pair $w=(g_1, g_2)\in A^2$ with $g_2\neq 0$, consider the associated Drinfeld module $\phi^w: A\rightarrow A\{\tau\}$ of rank $2$ defined by $\phi_T^w=T+g_1\tau+g_2\tau^2$. Fix integers $c_1, c_2\geq 1$ and define $|w|:=\op{max}\{|g_1|^{\frac{1}{c_1}}, |g_2|^{\frac{1}{c_2}}\}$. I show that when ordered by height, there is a positive density of pairs $w=(g_1, g_2)$, such that the $T$-adic Galois representation attached to $\phi^w$ is surjective.
\end{abstract}

\section{Introduction}
\subsection{Background and motivation}
\par Given an elliptic curve $E$ over $\Q$, and a prime natural number $n>1$, there is a Galois representation 
\[\rho_{E, n}: \op{Gal}(\bar{\Q}/\Q)\rightarrow \op{GL}_2(\Z/n\Z)\] on the $n$-torsion of $E(\bar{\Q})$. Moreover, given a prime $p$, the $p$-adic Tate-module $T_p(E):=\varprojlim_n E[p^n]$ admits a natural Galois action, which is encoded by the representation 
\[\hat{\rho}_{E, p}: \op{Gal}(\bar{\Q}/\Q)\rightarrow \op{GL}_2(\Z_p).\] Passing to the inverse limit of $\rho_{E,n}$, one obtains the \emph{adelic Galois representation}
\[\hat{\rho}_E: \op{Gal}(\bar{\Q}/\Q)\rightarrow \op{GL}_2(\widehat{\Z})\xrightarrow{\sim} \prod_p \op{GL}_2(\Z_p),\] which can also be identified with the product $\prod_p \hat{\rho}_{E, p}$. The celebrated \emph{open image theorem} of Serre \cite{serreopenimage} asserts that if $E$ does not have complex multiplication, then the image of $\hat{\rho}_E$ has finite index in $\op{GL}_2(\widehat{\Z})$. In particular, this means that for all but finitely many primes $p$, the $\rho_{E, p}:\op{Gal}(\bar{\Q}/\Q)\rightarrow \op{GL}_2(\Z/p\Z)$ is surjective. A prime $p$ is said to be \emph{exceptional} if $\rho_{E,p}$ is not surjective. Duke \cite{duke1997elliptic} showed that almost all elliptic curves have no exceptional primes (when ordered according to height). An elliptic curve $E_{/\Q}$ is said to be a \emph{Serre curve} if the index of $\hat{\rho}_E$ in $\op{GL}_2(\widehat{\Z})$ is equal to $2$. This is the minimal possible index, and Jones \cite{Jones} showed that almost all elliptic curves are Serre curves. Generalized results have been obtained for principally polarized abelian varieties over a rational base, cf. \cite{landesmanabelian}.

\par In this article I explore natural analogues of such questions for Drinfeld modules of rank $2$. These arithmetic objects are natural function field analogues of elliptic curves, and they give rise to compatible families of Galois representations. Drinfeld modules (and their generalizations) have come to prominence for their central role in the Langlands program over function fields. Let $q$ be a power of a prime number and $F$ be the rational function field $\F_q(T)$. Let $F^{\op{sep}}$ be a choice of separable closure of $F$. Take $A$ to be the polynomial ring $\F_q[T]$. Take $\widehat{A}:=\varprojlim_{\mathfrak{a}} A/\mathfrak{a}$, where $\mathfrak{a}$ runs over all non-zero ideals in $A$. Pink and R\"utsche \cite{pinkrutsche} proved an analogue of Serre's open image theorem for Drinfeld modules $\phi$ over $F$ without complex multiplication. In greater detail, suppose that $\phi: A\rightarrow F\{\tau\}$ is a Drinfeld module of rank $r$ for which $\op{End}_{\bar{F}}(\phi)=A$, then the image of the adelic Galois representation
\[\hat{\rho}_{\phi}: \op{Gal}(F^{\op{sep}}/F)\rightarrow \op{GL}_r(\widehat{A})\] has finite index in $\op{GL}_r(\widehat{A})$. For odd prime powers $q\geq 5$, Zywina \cite{zywina2011drinfeld} was in fact able to construct an explicit Drinfeld module of rank $2$ for which the adelic Galois representation is surjective. These results have been generalized to certain higher ranks by Chen \cite{Chen1,Chen2}. There has also been recent interest in the study of the mod-$T$ representation for certain infinite families of Drinfeld modules \cite{GowMcGuire}, as well as its application to the inverse Galois problem over the rational function field.
\subsection{Main result}
\par The strategy taken by Duke to prove that almost all elliptic curves have no expectional primes makes use of the large sieve. This involves local estimates that rely on bounds for certain sums of Hurwitz class numbers (in congruence classes). These sums are related to the fourier coefficients of certain modular forms, and the estimates rely on the Ramanujan bound (cf. \cite[Lemma 3]{duke1997elliptic}). The approach taken for Drinfeld modules in this article is indeed different, since Duke's method does not readily generalize to this context. I consider only the Galois representation at the rational prime $\p=(T)$. I then study the \emph{density} of rank $2$ Drinfeld modules $\phi$ for which the $T$-adic Galois representation 
\begin{equation}\label{T-adic Galrep}\hat{\rho}_{\phi, \p}:\op{Gal}(F^{\op{sep}}/F)\rightarrow \op{GL}_2(A_\p)\end{equation}is surjective. 

\begin{theorem}\label{main thm}
    Assume that $q\geq 5$ is odd. Then there is a positive density of Drinfeld modules of rank $2$ over $A:=\F_q[T]$ for which the $T$-adic Galois representation \eqref{T-adic Galrep} is surjective.
\end{theorem}

One may refer to section \ref{section 2.3}, where a precise notion of density is given for Drinfeld modules over $A$. \subsection{Organization} Including the introduction, this article consists of three sections. In section \ref{s 2}, I discuss the basic properties of Drinfeld modules and their associated Galois representations. In the final section, I prove that the surjectivity of the $T$-adic representation can indeed be detected from certain very explicit congruence conditions on the coefficients of a Drinfeld module with coefficients in $A$ (cf. Theorem \ref{main thm 0}). These conditions are very explicit, and they are leveraged to prove that the density of Drinfeld modules satisfying these conditions is positive. In fact, a lower bound for this density can be given, and I refer to the Remark \ref{only remark} for a discussion on the issues in obtaining a satisfactory lower bound. 
\subsection{Outlook} One of the key methods used is the algorithm of Gekeler for computing characteristic polynomials of rank $2$ Drinfeld modules. I would expect that they can be generalized to prove similar results for the Galois representations on $\p$-adic Tate modules, when $\p$ is a prime with small degree. However, this has not been pursued since one does not have a general statement that applied to all $\p$, and one only considers $\p=(T)$ in this article. I am hopeful that Theorem \ref{main thm} shall pave the way for interesting developments in the future.
\section{Drinfeld modules and associated Galois representations}\label{s 2}
\subsection{Preliminary notions}
\par I discuss the theory of Drinfeld modules; my notation is consistent with that in \cite{papibook}. Let $p$ be an odd prime number and $q=p^n$. Set $\F_q$ to be the field with $q$ elements and assume that $q\geq 5$. Let $A$ be the polynomial ring $\F_q[T]$ and $F$ its fraction field $\F_q(T)$. Given a non-zero ideal $\mathfrak{a}$ with generator $a$, set $\op{deg}\mathfrak{a}:=\op{deg}_T(a)$. Let $K$ be a field extension of $\F_q$, $K$ is called an $A$-field if it is equipped with an $\F_q$-algebra homomorphism $\gamma: A\rightarrow K$. Given an $A$-field $K$, the $A$-characteristic of $K$ is defined as follows
\[\op{char}_A(K):=\begin{cases}
    0 &\text{ if }\gamma\text{ if injective,}\\
    \op{ker}\gamma & \text{ otherwise.}
\end{cases}\]
Then $K$ is of \emph{generic characteristic} if $\op{char}_A(K)=0$. Take $K^{\op{sep}}$ to be a choice of separable closure of $K$, and set $\op{G}_K:=\op{Gal}(K^{\op{sep}}/K)$. The rational function field $F$ is an $A$-field of generic characteristic where the map $\gamma: A\hookrightarrow F$ is the natural inclusion map. Given a non-zero prime $\lambda$ of $A$, take $k_\lambda$ to denote the residue field of $A$ at $\lambda$. Then, $k_\lambda$ is an $A$-field of characteristic $\lambda$, where $\gamma_\lambda: A\rightarrow k_\lambda$ is the mod-$\lambda$ reduction map.  

\par Given an $\F_q$-algebra $K$, set $K\{\tau\}$ to denote the noncommutative ring of twisted polynomials over $K$, the elements of which are polynomials $f(\tau)=\sum_{i=0}^d a_i \tau^i$, with coefficients $a_i\in K$. Addition is defined \emph{term-wise} as follows
\[\begin{split}& \sum_i a_i \tau^i+\sum_i b_i \tau^i:=\sum_i(a_i+b_i) \tau^i,\\
& (a\tau^i) (b \tau^j):=ab^{q^i} \tau^{i+j}.\end{split}\] Given $f=\sum_{i=0}^d a_i \tau^i$, set
\[f(x):=\sum_{i=0}^d a_i x^{q^i}.\] The polynomial $f(x)$ is an $\F_q$-linear polynomial, i.e., given $c_1, c_2\in \F_q$, the following algebraic relation holds
\[f(c_1 x+c_2 y)=c_1 f(x)+c_2 f(y)\] in $K[x,y]$. I take $K\langle x \rangle $ to denote non-commutative ring of $\F_q$-linear polynomials defined by the relations
\[(f+g)(x):=f(x)+g(x)\text{ and }(f\cdot g)(x):=f(g(x)). \]
The map $f(\tau)\mapsto f(x)$ is an isomorphism of $\F_q$-algebras from $K\{\tau\}$ to $K\langle x\rangle$. 

\par Next, I define the notion of \emph{height} $\op{ht}_\tau (f)$ and \emph{degree} $\op{deg}_\tau (f)$ of a twisted polynomial $f(\tau)\in K\{\tau\}$. Write
\[f(\tau)=a_h \tau^h+a_{h+1} \tau^{h+1}+\dots +a_d \tau^d,\] where $a_h, a_d\neq 0$, and set 
\[\op{ht}_\tau (f):=h\text{ and }\op{deg}_\tau (f):=d.\]
Thus, one finds that the degree of $f(x)$ as a polynomial in $x$ is $q^{\op{deg}_\tau(f)}$. When referring to the degree of the polynomial in $x$, I write 
\[\op{deg}_x f(x):= q^{\op{deg}_\tau(f)}.\]Let $\partial:K\{\tau\}\rightarrow K$ be the \emph{derivative map} sending 
\[\sum_n a_n \tau^n\mapsto a_0.\] Thus, $\op{ht}(f)=0$ if and only if $\partial(f)\neq 0$. Note that $f(x)$ is separable if and only if its derivative in $x$ does not vanish. Since $\partial f=\frac{d}{dx} (f(x))$, one finds that $f(x)$ is separable if and only if $\op{ht}(f)=0$. With these conventions at hand, I recall the definition of a Drinfeld module of rank $r\geq 1$. In this article, I shall only be concerned with the case in which $r=2$. 
\begin{definition}
    Let $r\geq 1$ be an integer and $K$ be an $A$-field. A \emph{Drinfeld module} of rank $r$ over $K$ is a homomorphism of $\F_q$-algebras 
    \[\phi: A\rightarrow K\{\tau\},\] sending $a\in A$ to $\phi_a\in K\{\tau\}$, such that 
    \begin{itemize}
        \item $\partial(\phi_a)=\gamma(a)$ for all $a\in A$; 
        \item $\op{deg}_\tau(\phi_a)=r\op{deg}_T(a)$. 
    \end{itemize}
\end{definition}
The second condition is equivalent to requiring that $\op{deg}_\tau(\phi_T)=r$. I shall write 
\[\phi_T=T+g_1 \tau +g_2\tau^2 +\dots +g_r \tau^r,\] and observe that since $\op{deg}_\tau \phi_T=r$, $g_r\neq r$. The coefficients $(g_1, \dots, g_r)$ completely describe the module $\phi$.
\par Let $\phi$ and $\psi$ be Drinfeld modules, a morphism $u:\phi\rightarrow \psi$ is a function $u(\tau)\in K\{\tau\}$ such that $u \phi_a=\psi_a u$ for all $a\in A$. An isogeny is a non-zero morphism. It is easy to see that if there exists an isogeny, then, the rank of $\phi$ equals the rank of $\psi$. The group of all morphisms from $\phi$ to $\psi$ is denoted $\op{Hom}_K(\phi , \psi)$. An isomorphism between Drinfeld module is an isogeny with an inverse. It is easy to see that the only invertible elements in $K\{\tau\}$ are the non-zero constant elements $c\in K$. Write $\phi_T=T+g_1\tau+\dots +g_r \tau^r$ and $\psi_T=T+g_1'\tau+\dots+g_r' \tau^r$, one finds that $\phi$ is isomorphic to $\psi$ if and only if there is a non-zero constant $c$ such that 
\[g_i'=c^{q^i-1} g_i,\]
for $i\in [1, r]$. When $r=2$, the $j$-invariant of $\phi$ is defined as 
\[j(\phi):=\frac{g_1^{q+1}}{g_2}.\]
Two rank $2$ Drinfeld modules $\phi$ and $\phi'$ are isomorphic if and only if $j(\phi)=j(\phi')$ (cf. \cite[Ch. 3]{papibook}). 
\par Let $\Omega_F^\ast$ denote the set of discrete valuations of $F$, and $\Omega_F$ be the valuations corresponding to non-zero prime ideals of $A$. For $\lambda\in \Omega_F$, set $A_\lambda$ denote the completion of $A$ at $\lambda$, and $F_\lambda$ be its field of fractions. Taking $\mathfrak{m}_\lambda$ to be the maximal ideal of $A_\lambda$, set $k_\lambda:=A_\lambda/\mathfrak{m}_\lambda$. Let $\phi$ be a Drinfeld module over $F$ and $\lambda\in \Omega_F$. I denote by $\phi_\lambda$ the localized Drinfeld module over $F_\lambda$, defined to be the composite
\[\phi_\lambda: A\xrightarrow{\phi} F\{\tau\}\rightarrow F_\lambda\{\tau\},\] where the second map is induced by the natural inclusion of $F$ into $F_\lambda$. Say that $\phi$ has \emph{stable reduction} at $\lambda$ if there is a Drinfeld module $\psi$ over $F_\lambda$ that is isomorphic to $\phi_\lambda$ with coefficients in $A_\lambda$, for which the reduction \[\bar{\psi}:A\rightarrow k_\lambda\{\tau\}\] is a Drinfeld module. The rank of $\bar{\psi}$ is referred to as the \emph{reduction rank} of $\phi$ at $\lambda$. If the reduction rank is $r$, then $\phi$ has \emph{good reduction} at $\lambda$. Let $\lambda\in \Omega_F$ be a prime of good reduction, and consider the reduced Drinfeld module
\[\bar{\psi}:A\rightarrow k_\lambda\{\tau\}.\] Let $a$ be the monic generator of $\lambda$. The element $a$ is in the kernel of the reduction map $\gamma: A\rightarrow k_\lambda$. The constant term of $\phi_a$ is $\gamma(a)=0$, and thus, $\op{ht}_\tau(\phi_a)>0$. In fact, $\op{deg}_T(a)$ divides $\op{ht}_\tau(\phi_a)$ and the \emph{height of $\bar{\psi}$} is defined to be the integer
\[H(\bar{\psi}):=\frac{\op{ht}_\tau(\phi_a)}{\op{deg}_T(a)},\] which one refers to as the height of the reduction of $\phi$ at $v$. It is easy to see that $H(\bar{\psi})\in [1, r]$, I refer to \cite[section 3.2]{papibook} for further details. The quantity $H(\bar{\psi})$ is called the height of the reduction of $\phi$ at $\lambda$.

\par Given a Drinfeld module $\phi$ of rank $r$ over $F$, there is an isomorphic Drinfeld module $\psi$ for which 
\[\psi_T=T+g_1 \tau +g_2 \tau^2+\dots g_r \tau^r\] with coefficients $g_i\in A$. A tuple $(g_1, \dots, g_r)\in A^r$ is said to be \emph{minimal} if there is no non-constant polynomial $f\in A$ such that $g_i$ is divisible by $c^{q^i-1}$ for all $i\in [1, r]$. It is easy to see that $\psi$ can be uniquely chosen so that $(g_1, \dots, g_r)$ is minimal in the above sense. Moreover, the primes of bad reduction are precisely those that divide the leading coefficient $g_r$. It is thus clear that there are only finitely many primes of bad reduction. The primes of unstable reduction are those that divide every single coefficient $g_i$ for $i\in [1, r]$. 

\subsection{Galois representations}
\par I discuss Galois representations associated to Drinfeld modules over $F$ or rank $r$. Let $\phi$ be a Drinfeld module over $F$ and $a$ be a non-zero polynomial in $A$. Since the constant coefficient of $\phi_a$ is $\gamma(a)=a$, it follows that $\phi_a(x)$ is a separable polynomial. Let $\phi[a]\subset F^{\op{sep}}$ denote the set of roots of $\phi_a(x)$. I note that if $b\in A$, then, 
\[\phi_a(\phi_b(x))=\phi_{ab}(x)=\phi_b(\phi_a (x))=0.\]One defines a twisted $A$-module structure on $F^{\op{sep}}$ by $b\ast \alpha:=\phi_b(\alpha)$. Then, the above computation shows that $\phi[a]$ is an $A$-submodule of $F^{\op{sep}}$. Given a non-zero ideal $\mathfrak{a}$ in $A$, set $\phi[\mathfrak{a}]:=\phi[a]$, where $a$ is the monic generator of $\mathfrak{a}$. Then, one finds that $\phi[\mathfrak{a}]\simeq (A/\mathfrak{a})^r$. The Galois group $\op{G}_F:=\op{Gal}(F^{\op{sep}}/F)$ naturally acts on $\phi[\mathfrak{a}]$ by $A$-module automorphisms. I denote by \[\rho_{\phi, \mathfrak{a}}:\op{G}_F\rightarrow \op{Aut}_A(\phi[\mathfrak{a}])\xrightarrow{\sim}\op{GL}_r(A/\mathfrak{a})\] the associated Galois representation. Let $\p$ be a non-zero prime ideal of $A$ and take $T_\p(\phi)$ to denote the $\p$-adic Tate-module, defined to be the inverse limit
\[T_\p(\phi):=\varprojlim_n \phi[\p^n].\]
Choosing a $A_{\p}$-basis of $T_{\p}(\phi)$ the associated Galois representation is denoted 
\[\hat{\rho}_{\phi, \p}: \op{G}_F\rightarrow \op{Aut}_{A_\p} (T_\p(\phi))\xrightarrow{\sim} \op{GL}_r(A_\p).\]
The mod-$\p$ reduction of $\hat{\rho}_{\phi, \p}$ is identified with $\rho_{\phi, \p}$. The Galois representation $\rho_{\phi, \mathfrak{a}}$ is unramified at all primes $\mathfrak{l}\nmid \mathfrak{a}$ of $A$ at which $\phi$ has good reduction. Moreover, given a non-zero prime $\mathfrak{l}\neq \mathfrak{p}$ of $A$, the representation $\hat{\rho}_{\phi, \p}$ is unramified at $\mathfrak{l}$ if and only if $\mathfrak{l}$ is a prime of good reduction for $\phi$. 

\par Assume that $\phi$ has good reduction at $\p$. Let me briefly discuss the structure of $T_\p(\phi)$ as a module over $\op{G}_{\p}:=\op{G}_{F_\p}$. Let $H$ denote the height of the reduction of $\phi$ at $\p$. Let $n$ be a positive integer and $v=v_\p$ be the valuation on $F_\p$ normalized by $v(\pi_v)=1$, where $\pi_v$ is a uniformizer of $F_\p$. \[\phi[\p^n]^0:=\{\alpha\in \phi[\p^n]\mid v(\alpha)>0\}.\] For ease of notation, denote by $\bar{\phi}$ the reduction of $\phi$ at $\p$, and $\bar{\phi}[\p^n]$ be its $\p^n$-torsion module. Then, $\bar{\phi}[\p^n]$ is an unramified $\op{G}_\p$-module and can be identified with the quotient $\phi[\p^n]/\phi[\p^n]^0$. One has a short exact sequence of $\op{G}_{\p}$-modules
\[0\rightarrow \phi[\p^n]^0\rightarrow \phi[\p^n]\rightarrow \bar{\phi}[\p^n]\rightarrow 0,\]cf. \cite[(6.3.1) on p. 371]{papibook}. Taking \[T_\p(\phi)^0:=\varprojlim_n \phi[\p^n]^0\text{ and }T_\p(\bar{\phi}):=\varprojlim_n \bar{\phi}[\p^n],\] and obtain a short exact sequence 
\[0\rightarrow T_\p(\phi)^0\rightarrow T_\p(\phi)\rightarrow T_\p(\bar{\phi})\rightarrow 0,\] cf. \cite[Proposition 6.3.9]{papibook}.
As an $A_\p$-module, \[\phi[\p^n]^0\simeq (A/\p^n)^H\text{, and }\bar{\phi}[\p^n]\simeq (A/\p^n)^{r-H}.\] Thus, one finds that 
\[T_\p(\phi)^0\simeq A_\p^H\text{ and } T_\p(\bar{\phi})\simeq A_\p^{r-H}.\]
Let $F_\p(\phi[\p^n]^0)$ be the field extension of $F_\p$ generated by $\phi[\p^n]^0$. In other words, it is the fixed field of the kernel of the action map
\[\varrho: \op{G}_\p\rightarrow \op{Aut}(\phi[\p^n]^0)\xrightarrow{\sim} \op{GL}_H(A/\p^n).\]

\begin{lemma}\label{ramification count lemma}
    With respect to notation above, the ramification index of $F_\p(\phi[\p^n]^0)/F_\p$ is divisible by $q^{N(n-1)}(q^N-1)$, where $N:=H\op{deg}(\p)$.
\end{lemma}
\begin{proof}
    The result is \cite[Lemma 6.3.10]{papibook}.
\end{proof}
\begin{proposition}\label{det surjective prop}
    With respect to the above notation, assume that $H=1$. Then, the determinant map 
    \[\op{det}\hat{\rho}_{\phi, \p}: \op{G}_F\rightarrow A_\p^\times\] is surjective.
\end{proposition}
\begin{proof}
    It suffices to show that for all $n>0$, the determinant map 
     \[\op{det}\rho_{\phi, \p^n}: \op{G}_F\rightarrow \left(A/\p^n\right)^\times\] is surjective. Consider the restriction of $\op{det}\rho_{\phi, \p^n}$ to the inertia group $\op{I}_\p$ of $\op{G}_\p$. Let $\psi:\op{I}_\p\rightarrow A_\p^\times$ be the character for the action of the inertia group on $T_\p(\phi)^0$ and $\psi_n$ be its mod-$\p$ reduction. It follows from Lemma \ref{ramification count lemma} that $\psi_n$ is surjective for all $n$. Since $\bar{\phi}[\p^n]$ is unramified at $\p$, it follows that $\op{det}\hat{\rho}_{\phi, \p}$ coincides with $\psi$, when restricted to $\op{I}_\p$. It follows that 
    \[\op{det}\hat{\rho}_{\phi, \p}(\op{I}_\p)=\psi(\op{I}_\p)=A_\p^\times .\] This shows that the determinant map $\op{det}\hat{\rho}_{\phi, \p}$ is surjective.
\end{proof}
\subsection{Counting Drinfeld modules}\label{section 2.3}
\par In this section, I consider the data defining a Drinfeld module over $A$ of rank $2$. I say that $\phi$ is defined over $A$ if 
\[\phi_T=T+g_1\tau +g_2\tau^2,\] where $(g_1, g_2)\in A^2$. It is easy to see that any Drinfeld module over $F$ is isomorphic to a Drinfeld module defined over $A$. Since I assume that $\phi$ has rank $2$, it follows that $g_2\neq 0$. The pair $w=(g_1, g_2)\in A^2$ with $g_2\neq 0$ shall be referred to as a \emph{Drinfeld datum}. I shall take $\phi^w$ to be the Drinfeld module defined by
\[\phi^w_T=T+g_1\tau +g_2\tau^2.\] The pair $w$ does not uniquely determine the Drinfeld module $\phi$ up to isomorphism over $F$. Take $\mathcal{C}$ to be the set of Drinfelf data $w=(g_1, g_2)\in A^2$. 
\par Fix a pair of positive integers $\vec{c}=(c_1, c_2)$. Given a polynomial $f$, set $|f|:=q^{\op{deg}(f)}$. Then, the height 
\[|w|:=\op{max}\left\{|g_1|^{\frac{1}{c_1}}, |g_2|^{\frac{1}{c_2}}\right\}.\]
Let $X>0$ be an integer, take $\cC(X)$ to be the set of Drinfeld modules with $|\phi|< q^X$. Then, this is the set of Drinfeld data $w=(g_1, g_2)$ such that  
\[\deg g_1<  c_1 X\text{ and }\deg g_2< c_2 X.\]
Thus, one finds that 
\begin{equation}
    \# \cC(X)=q^{c_1X}(q^{c_2X}-1).
\end{equation}

Let $\mathcal{S}$ be a subset of $\mathcal{C}$ and given a positive integer $X>0$, take $\mathcal{S}(X):=\{w\in \mathcal{S}\mid |w|<X\}$. Refer to the following limit 
\[\mathfrak{d}(\mathcal{S})=\lim_{X\rightarrow \infty} \frac{\# \mathcal{S}(X)}{\# \mathcal{C}(X)}=\frac{\# \mathcal{S}(X)}{q^{c_1X}(q^{c_2X}-1)},\]
(provided it exists) as the \emph{density of $\mathcal{S}$}. The upper and lower densities $\overline{\mathfrak{d}}(\mathcal{S})$ and $\underline{\mathfrak{d}}(\mathcal{S})$ are defined as follows
\[\begin{split}
    &\overline{\mathfrak{d}}(\mathcal{S}):=\limsup_{X\rightarrow \infty} \left(\frac{\# \mathcal{S}(X)}{\# \mathcal{C}(X)}\right);\\
    &\underline{\mathfrak{d}}(\mathcal{S}):=\liminf_{X\rightarrow \infty} \left(\frac{\# \mathcal{S}(X)}{\# \mathcal{C}(X)}\right),
\end{split}\]
and unlike $\mathfrak{d}(\mathcal{S})$, these densities are guaranteed to exist.
One says that $\mathcal{S}$ has positive density if $\underline{\mathfrak{d}}(\mathcal{S})>0$. 

\begin{remark}
    Recall that a pair $w=(g_1, g_2)\in A^2$ with $g_2\neq 0$ is said to be minimal if there is no non-constant polynomial $a$ such that $a^{q-1}$ divides $g_1$ and $a^{q^2-1}$ divides $g_2$. One may instead define $\mathcal{C}_{\op{min}}$ to consist of only minimal pairs $w$, and define the density of a subset $\mathcal{S}$ of $\mathcal{C}_{\op{min}}$ by 
    \[\mathfrak{d}(\mathcal{S}):=\frac{\# \mathcal{S}(X)}{\# \mathcal{C}_{\op{min}}(X)}.\]
    It is possible to adopt this approach instead since it relies on point counting techniques in certain weighted projective spaces, cf. \cite[section 5]{phillips2022counting}. I do not however pursue it in this article since it makes calculations more cumbersome.
\end{remark}

\section{Main results}
\par In this section, the main results of the article are proven. I shall assume that $q\geq 5$ is odd. There are two distinct elements $a_1, a_2\in \F_q^\times$. For $i=1,2$, set $\lambda_i:=(T-a_i)$. I set $v_i$ to denote the valuation at $\lambda_i$, normalized by $v_i(T-a_i)=1$. Let $w=(g_1, g_2)$ be a Drinfeld datum for a rank $2$ Drinfeld module $\phi=\phi^w: A\rightarrow A\{\tau\}$. Throughout this section, I fix an element $\eta\in \F_q^\times$ which is not a square in $\F_q^\times$. This is possible since $q$ is assumed to be odd, and thus, $\F_q^\times$ has even order (and therefore not $2$-divisible). Set $\p:=(T)$ and for ease of notation, take $\rho$ to denote the $T$-adic representation $\hat{\rho}_{\phi^w, \p}: \op{G}_F\rightarrow \op{GL}_2(A_{\p})$. Set \[\bar{\rho}:\op{G}_F\rightarrow \op{GL}_2\left(A/\p\right)\xrightarrow{\sim} \op{GL}_2(\F_q)\] to denote the mod-$\p$ reduction of $\rho$. For $i=1, 2$, set $\op{G}_{\lambda_i}:=\op{Gal}(F_{\lambda_i}^{\op{sep}}/F_{\lambda_i})$ and $\op{I}_i\subset G_{\lambda_i}$ be the inertia group. The Frobenius in $\op{G}_{\lambda_i}/\op{I}_i$ is denoted $\sigma_i$. Likewise, $\op{G}_{\p}:=\op{Gal}(F_{\p}^{\op{sep}}/F_\p)$ and $\op{I}_{\p}$ is the inertia subgroup.

\subsection{Congruence criteria for surjective $T$-adic images}  In this subsection, we prove the following result.
\begin{theorem}\label{main thm 0}
    Let $\eta\in \F_q^\times$ be a non-square, and let $a_1, a_2$ be distinct non-zero elements in $\F_q^\times$. Consider a pair $(g_1, g_2)\in A^2$ with $g_2\neq 0$ such that 
    \begin{enumerate}
        \item $g_1$ is divisible by $(T-a_1)$,
        \item $g_1$ is coprime to $T(T-a_2)$,
        \item $g_2$ is divisible by $(T-a_2)$ and not by $(T-a_2)^2$,
        \item $g_2\equiv -a_1 \eta^{-1}\mod{(T-a_1)}$,
        \item $g_2$ is coprime to $T$.
    \end{enumerate}
    Then, the representation $\rho:\op{G}_F\rightarrow \op{GL}_2(A_{\p})$ is surjective.
\end{theorem}
I prove some preparatory results. Let $\pi$ be the uniformizer of $A_\p$ at $\p$. The congruence filtration on $\op{GL}_2(A_\p)$ is defined by 
\[\begin{split}
    & G_\p^0:=\op{GL}_2(A_\p) \text{ for }i=0;\\
     & G_\p^i:=\op{Id}+\pi^i\op{M}_2(A_\p)  \text{ for }i>0;
\end{split}\]
where $\op{M}_2(A_\p)$ denotes the $2\times 2$ matrices with entries in $A_\p$. Let $\mathcal{H}$ be a closed subgroup of $\op{GL}_2(A_\p)$, take $\mathcal{H}^i:=\mathcal{H}\cap G_\p^i$, and 
\[\mathcal{H}^{[i]}:=\mathcal{H}^i/\mathcal{H}^{i+1}.\]

\begin{proposition}\label{PR prop}
    Let $\mathcal{H}$ be a closed subgroup of $\op{GL}_2(A_\p)$. Assume that the following conditions hold
    \begin{enumerate}
        \item $q\geq 4$, 
        \item $\op{det}(\mathcal{H})=A_\p^\times$,
        \item $\mathcal{H}^{[0]}=\op{GL}_2(\F_q)$, 
        \item $\mathcal{H}^{[1]}$ contains a non-scalar matrix.
    \end{enumerate}
    Then one has that
    \[\mathcal{H}=\op{GL}_2(A_\p).\]
\end{proposition}
\begin{proof}
    This result is \cite[Proposition 4.1]{pinkrutsche}.
\end{proof}

\begin{lemma}\label{unipotent lemma}
Let $\phi: A\rightarrow F\{\tau\}$ be a Drinfeld module of rank $2$ associated to a pair $(g_1, g_2)\in A^2$, i.e.,
\begin{equation}\label{inertia image conditions}\phi_T=T+g_1\tau+g_2\tau^2.\end{equation}Let $\lambda=(a)$ be a non-zero prime ideal of $A$ and $v$ be the associated valuation normalized by setting $v(a)=1$. Take $\op{I}_v$ to be the inertia group of $\op{G}_{F_v}$. Assume that \[v(g_1)=0\text{ and }v(g_2)=1.\] Let $\mathfrak{a}$ be a proper ideal of $A$. Then there is a basis of $\phi[\mathfrak{a}]$ such that $\rho_{\phi, \mathfrak{a}}(\op{I}_v)$ contains the full group of unipotent upper triangular matrices $\mtx{1}{\ast}{0}{1}$ in $\op{GL}_2(A/\mathfrak{a})$.
\end{lemma}
\begin{proof}
  It follows from the conditions \eqref{inertia image conditions} that $\phi$ has stable reduction at $v$ with reduction rank equal to $1$. Moreover, the valuation of the $j$-invariant is $-1$. The result then follows from \cite[Proposition 4.1]{zywina2011drinfeld}\footnote{I am happy to reproduce the details here if the referee insists.}.
\end{proof}

\begin{lemma}\label{G contains SL_2 lemma}
    Let $\F$ be a finite field and $G$ be a subgroup of $\op{GL}_2(\F)$ such that 
    \begin{enumerate}
        \item $G$ contains a subgroup of order $\# \F$, 
        \item $G$ acts irreducibly.
    \end{enumerate}
    Then, $G$ must contain $\op{SL}_2(\F)$. 
\end{lemma}
\begin{proof}
    This result is \cite[Lemma A.1]{zywina2011drinfeld}.
\end{proof}

\begin{proposition}\label{main prop}
    Let $\phi$ be a Drinfeld module associated to a pair $(g_1, g_2)\in A^2$ such that 
    \begin{enumerate}
        \item $g_1$ is divisible by $(T-a_1)$,
        \item $g_1$ is coprime to $(T-a_2)$,
        \item $g_2$ is divisible by $(T-a_2)$ and not by $(T-a_2)^2$,
        \item $g_2\equiv -a_1 \eta^{-1}\mod{(T-a_1)}$.
    \end{enumerate}
    Then, the image of the representation \[\bar{\rho}:\op{G}_F\rightarrow \op{GL}_2(\F_q)\] contains $\op{SL}_2(\F_q)$.
\end{proposition}
\begin{proof}
    Let $G$ be the image of $\bar{\rho}$ and let $\bar{V}:=\phi[\p]\simeq \F_q^2$ be the underlying space on which $G$ acts. 
    \par First, I show that $G$ acts irreducibly on $\bar{V}$. In other words, there is no $\F_q$-basis of $\bar{V}$ with respect to which $G$ consists only of upper triangular matrices $\mtx{\ast}{\ast}{0}{\ast}$. Since $g_2\equiv -a_1\eta^{-1}\mod{(T-a_1)}$, and both $a_1$ and $\eta$ are non-zero constants in $\F_q$, one finds that $(T-a_1)$ does not divide $g_2$. As a consequence, $\phi$ has good reduction at $\lambda_1$. Recall that 
    \[\rho:\op{G}_F\rightarrow \op{GL}_2(A_\p)\] is the Galois representation on the $T$-adic Tate-module $T_\p(\phi)$. This representation is unramified at $\lambda_1$, and thus there is a well defined element $\rho(\sigma_1)\in \op{GL}_2(A_\p)$. Let $a, b\in A_\p$ be coefficients of the characteristic polynomial of $\rho(\sigma_1)$, i.e., 
    \[a:=\op{trace}\rho(\sigma_1)\text{ and }b:=\op{det}\rho(\sigma_1).\]
    The elements $a, b$ are in the image of the natural embedding $A\hookrightarrow A_\p$, and can thus be identified with elements in $A$ (cf. \cite[Section 4.2, p.236]{papibook}). These elements $a,b$ depend only on $\lambda_1$ and not on the prime $\p$. Moreover, it follows from \cite[Theorem 4.2.7, (2)]{papibook} that 
    \[\op{deg}_T(a)=0\text{ and }\op{deg}_T(b)\leq 1.\]
    I compute $a$ and $b$ using Gekeler's algorithm for Drinfeld modules of rank $2$, cf. \cite[p.248, l.-2]{papibook}. This this context, Gekeler's result simply states that 
    \[\begin{split}
       & a= -\bar{g}_2^{-1} \bar{g}_1 \mod{\lambda_1};\\
       & b=-\bar{g}_2^{-1}(T-a_1),
    \end{split}\]
    where $\bar{g}_i:=g_i\mod{\lambda_1}$ is a constant in $\F_q$. Note that the first of the above equations determines $a$, since $a$ is a constant, and $\lambda_1$ has degree $1$. I warn the reader that $a$ and $b$ are computed only with reference to $\lambda_1$ and the above computation has no dependence on $\p=(T)$. After computing $a$ and $b$, I shall reduce them modulo $(T)$. Since $g_1$ is divisible by $(T-a_1)$, one finds that $\bar{g}_1=0$. On the other hand, $\bar{g}_2=-a_1\eta^{-1}$. Hence, one finds that
    \[a=0\text{ and } b=\eta (a_1^{-1}T-1).\]

     Let $\bar{a}, \bar{b}\in \F_q$ denote the mod-$T$ reductions of $a$ and $b$ respectively. Then, one has that 
    \[\bar{a}:=\op{trace}\bar{\rho}(\sigma_1)\text{ and }\bar{b}:=\op{det}\bar{\rho}(\sigma_1),\] where I remind the reader that $\bar{\rho}$ is the mod-$T$ reduction of $\rho$.
    Note that $\bar{a}=0$ and $\bar{b}=-\eta$. Thus, the trace and determinant of $\bar{\rho}(\sigma_1)$ are $0$ and $-\eta$ respectively. Suppose that by way of contradiction that there is an $\F_q$-basis of $\bar{V}$ with respect to which $\bar{\rho}(\sigma_1)$ is upper triangular. Then, since the trace of this matrix is $0$, one has that 
    \[\bar{\rho}(\sigma_1)=\mtx{c}{d}{0}{-c},\] and the determinant is $-c^2=-\eta$. Thus, $\eta$ is a square in $\F_q^\times$, which is a contradiction. Therefore, $G$ acts on $\bar{V}$ irreducibly. 
    \par In order to complete the proof, note that the assumptions imply that $v_2(g_1)=0$ and $v_2(g_2)=1$. Taking $\mathfrak{a}:=\p$ and $\lambda:=\lambda_2$ in Lemma \ref{unipotent lemma}, one finds that the full group of unipotent upper triangular matrices is in $G$. Thus in particular, $G$ contains a subgroup with $q$ elements. I have shown that the hypotheses of Lemma \ref{G contains SL_2 lemma} are satisfied, and thus, $G$ contains $\op{SL}_2(\F_q)$.
\end{proof}
 I now give the proof of Theorem \ref{main thm 0}.
\begin{proof}[Proof of Theorem \ref{main thm 0}]
    Let $\bar{\rho}$ be the mod-$\p$ reduction of $\rho$. It follows from Proposition \ref{main prop} that the image of $\bar{\rho}$ contains $\op{SL}_2(\F_q)$. It is assumed that $g_2$ is coprime to $T$, and hence, $\phi$ has good reduction at $T$. Since $g_1$ is coprime to $T$, the height of the reduction of $\phi$ at $\p$ is $1$. It then follows from Proposition \ref{det surjective prop} that $\op{det}\rho$ is surjective. Let $\rho_2$ denote the reduction of $\rho$ modulo $\p^2$, and let $\mathcal{G}\subseteq \op{GL}_2(A/\p^2)$ be its image. In order to complete the proof, it suffices to show that $\mathcal{G}$ contains a non-scalar matrix that is identity modulo $\p$. Applying Lemma \ref{unipotent lemma} to $\lambda=\lambda_2$ and $\mathfrak{a}:=\p^2$, one finds that $\mathcal{G}$ contains the full upper triangular unipotent subgroup of $\op{GL}_2(A/\mathfrak{p}^2)$ (up to conjugation). Thus in particular, it contains a nontrivial unipotent matrix that reduces to the identity modulo-$\p$. Now set $\mathcal{H}$ to denote the image of $\rho$. I have shown that the conditions of Proposition \ref{PR prop} are satisfied for $\mathcal{H}$, and thus, I conclude from it that $\mathcal{H}=\op{GL}_2(A_\p)$. 
\end{proof}

\subsection{Proof of the main result}
In this short subsection, we give a proof of the main result from the Introduction.
\begin{proof}[Proof of Theorem \ref{main thm}]
    Fix an element $\eta\in \F_q^\times$ which is not contained in $(\F_q^\times)^2$. Let $w=(g_1, g_2)\in A^2$ be a pair satisfying the congruence conditions of Theorem \ref{main thm 0}, then the $T$-adic Galois representation associated to $\phi^w$ is surjective. In other words, one requires that $g_1=b_1\mod{T(T-a_1)(T-a_2)}$ where $b_1$ is divisible by $(T-a_1)$, but is coprime to $T(T-a_2)$, and that $g_2=b_2\mod{T(T-a_1)(T-a_2)}$ where $b_2$ is divisible by $(T-a_2)$ but not by $(T-a_2)^2$, is coprime to $T$ and $b_2\equiv -a_1\eta^{-1} \mod{(T-a_1)}$. Fix one such pair $(b_1, b_2)$ represented by $b_i$ with 
    \[\op{deg}_T b_1<3\text{ and }\op{deg}_T b_2<4.\] Let $\mathcal{S}$ be the set of all pairs $w=(g_1, g_2)$ such that 
    \[g_1\equiv b_1\mod{T(T-a_1)(T-a_2)}\text{ and }g_2\equiv b_2\mod{T(T-a_1)(T-a_2)^2}.\] Let $X>0$ be an integer such that $c_1 X, c_2 X\geq 4$. Then, the polynomials $g_1$ for which $\op{deg} g_1< c_1 X$ and such that $g_1\equiv b_1\mod{T(T-a_1)(T-a_2)}$ can be written uniquely as $g_1=b_1 + a_1 T(T-a_1)(T-a_2)$, where $\op{deg} a_1< c_1 X -3$. Thus, the number of such polynomials $g_1$ is exactly equal to $q^{c_1 X -3}$. Likewise, the number of polynomials $g_2=b_2+a_2 T(T-a_1)(T-a_2)^2$ is exactly equal to $q^{c_2 X-4}$. Therefore, one finds that for $X\geq 4\op{max}\{c_1^{-1}, c_2^{-1}\}$, 
    \[\# \mathcal{S}(X)=q^{(c_1+c_2) X-7}.\] Thus, I have that $\lim_{X\rightarrow \infty} \frac{\# \mathcal{S}(X)}{\# \mathcal{C}(X)}=q^{-7}$. Thus in particular, $\mathcal{S}$ has positive density. Let $\mathcal{T}$ be the set of all pairs $w\in \mathcal{C}$ such that the $T$-adic representation attached to $\phi^w$ is surjective. Then, $\mathcal{T}$ contains $\mathcal{S}$ and 
    \[\underline{\mathfrak{d}}(\mathcal{T})\geq \underline{\mathfrak{d}}(\mathcal{S})=\mathfrak{d}(\mathcal{S})>0.\]
    This completes the proof of the main result.
\end{proof}
\begin{remark}\label{only remark}
The density is calculated to be $q^{-7}$ for a choice of $5$-tuple $(b_1, b_2, \lambda_1, \lambda_2, \eta)$. Two choices $(b_1, b_2, \lambda_1, \lambda_2, \eta)$ and $(b_1', b_2', \lambda_1', \lambda_2', \eta')$ could have a non-empty intersection. I have not made an effort to obtain a closed formula for the density of pairs $(g_1, g_2)$ that satisfy the union of conditions imposed for all such $5$-tuples. 
\end{remark}
\bibliographystyle{alpha}
\bibliography{references}
\end{document}